\documentclass[12pt]{amsart}
\usepackage[utf8]{inputenc}
\usepackage[english]{babel}
\usepackage{amsmath,amssymb,a4wide,latexsym,amscd,amsthm}
\usepackage{graphics}

\theoremstyle{plain}
\newtheorem{theorem}{Theorem}[section]
\newtheorem{lemma}{Lemma}[section]
\newtheorem{proposition}{Proposition}[section]
\newtheorem{corollary}{Corollary}[section]

\theoremstyle{definition}

\newtheorem{remark}{Remark}[section]

\numberwithin{equation}{section}

\begin{document}

\title[]
{On the optimal error bound for the first step in the method
of cyclic alternating projections}

\author[I. Feshchenko]{Ivan Feshchenko}
\address{Taras Shevchenko National University of Kyiv,
Faculty of Mechanics and Mathematics, Kyiv, Ukraine and
Samsung R\&D Institute Ukraine, 57 L'va Tolstogo str., Kiev 01032, Ukraine}
\email{ivanmath007@gmail.com and i.feshchenko@samsung.com}

\begin{abstract}
Let $H$ be a Hilbert space and $H_1,...,H_n$ be closed subspaces of $H$.
Set $H_0:=H_1\cap H_2\cap...\cap H_n$ and let 
$P_k$ be the orthogonal projection onto $H_k$, $k=0,1,...,n$.
The paper is devoted to the study of functions $f_n:[0,1]\to\mathbb{R}$
defined by
$$
f_n(c)=\sup\{\|P_n...P_2 P_1-P_0\|\,|c_F(H_1,...,H_n)\leqslant c\},\,c\in[0,1],
$$
where the supremum is taken over all systems of subspaces $H_1,...,H_n$
for which the Friedrichs number $c_F(H_1,...,H_n)$ is less than or equal to $c$.
Using the functions $f_n$ one can easily get an upper bound for the rate of convergence
in the method of cyclic alternating projections.
We will show that the problem of finding $f_n(c)$ is equivalent to 
a certain optimization problem on a subset of the set of
Hermitian complex $n\times n$ matrices.
Using the equivalence we find $f_3$ and study properties of $f_n$, $n\geqslant 4$.
Moreover, we show that
$$
1-a_n(1-c)-\widetilde{b}_n(1-c)^2\leqslant 
f_n(c)\leqslant 1-a_n(1-c)+b_n(1-c)^2
$$
for all $c\in[0,1]$,
where $a_n=2(n-1)\sin^2(\pi/(2n))$, $b_n=6(n-1)^2\sin^4(\pi/(2n))$
and $\widetilde{b}_n$ is some positive number.
\end{abstract}



\maketitle

\section{Introduction}

\subsection{The Friedrichs number of a pair of subspaces and
the method of alternating projections for two subspaces}
Let $H$ be a complex Hilbert space and $H_1,H_2$ be two closed subspaces of $H$.
The number $c_F(H_1,H_2)$ defined by
\begin{equation*}
c_F(H_1,H_2):=\sup\{|\langle x,y\rangle|\,| 
x\in H_1\ominus (H_1\cap H_2), \|x\|\leqslant 1,
y\in H_2\ominus (H_1\cap H_2), \|y\|\leqslant 1\}
\end{equation*}
is called the Friedrichs number (more precisely, the cosine of the Friedrichs angle)
of subspaces $H_1,H_2$.
Why is $c_F$ important?
A few properties of a pair $H_1,H_2$ can be formulated in terms of the Friedrichs number,
for example
\begin{enumerate}
\item
the orthogonal projections onto $H_1$ and $H_2$ commute if and only if $c_F(H_1,H_2)=0$;
\item
the sum $H_1+H_2$ is closed if and only if $c_F(H_1,H_2)<1$,
\end{enumerate}
see, e.g., \cite{Deu95}.
Also, the Friedrichs number is closely related to the rate of convergence in the
method of alternating projections.
This is a well-known method of finding the orthogonal projection of a given element $x\in H$ 
onto the intersection $H_1\cap H_2$ when the orthogonal projections $P_1$ and $P_2$
onto $H_1$ and $H_2$ are assumed to be known.
Define the sequence $x_0:=x$, $x_1:=P_1 x_0$, $x_2:=P_2 x_1$,
$x_3:=P_1 x_2$, $x_4:=P_2 x_3$ and so on.
Back in 1933 von Neumann \cite{Neu50} 
proved that $x_k\to P_0 x$ as $k\to\infty$,
where $P_0$ is the orthogonal projection onto $H_1\cap H_2$.
What can be said about the rate of convergence?
Since $x_{2k}=(P_2 P_1)^k x$, we see that
\begin{equation*}
x_{2k}-P_0 x=((P_2 P_1)^k-P_0)x.
\end{equation*}
With respect to the orthogonal decomposition 
$H=(H_1\cap H_2)\oplus (H\ominus (H_1\cap H_2))$ we have
$P_1=I\oplus P_1'$, $P_2=I\oplus P_2'$ and $P_0=I\oplus 0$,
where $I$ is the identity operator and $P_1',P_2'$ are orthogonal projections.
Hence 
\begin{equation*}
(P_2 P_1)^k-P_0=0\oplus(P_2' P_1')^k=(P_2 P_1-P_0)^k
\end{equation*} 
and
\begin{equation*}
\|x_{2k}-P_0 x\|=\|((P_2 P_1)^k-P_0)x\|=
\|(P_2 P_1-P_0)^k x\|\leqslant\|P_2 P_1-P_0\|^k \|x\|.
\end{equation*}
But $\|P_2 P_1-P_0\|=c_F(H_1,H_2)$ (see, e.g., \cite{Deu95}) and therefore we get estimate
\begin{equation*}
\|x_{2k}-P_0 x\|\leqslant (c_F(H_1,H_2))^k\|x\|.
\end{equation*}
This estimate is not sharp.
Aronszajn \cite{Aro50} proved that 
$$\|(P_2 P_1)^k-P_0\|\leqslant (c_F(H_1,H_2))^{2k-1}.$$
Therefore we get
\begin{equation*}
\|x_{2k}-P_0 x\|\leqslant (c_F(H_1,H_2))^{2k-1}\|x\|.
\end{equation*}
It is worth mentioning that this estimate is sharp because Kayalar and Weinert \cite{KW88} proved that 
$$\|(P_2 P_1)^k-P_0\|=(c_F(H_1,H_2))^{2k-1},\,k\geqslant 1.$$

\subsection{The method of cyclic alternating projections for $n$ subspaces}
Let $H$ be a complex Hilbert space and $H_1,...,H_n$ be closed subspaces of $H$.
The method of cyclic alternating projections is a well-known method
of finding the orthogonal projection of a given element $x\in H$ 
onto the intersection $H_1\cap H_2\cap...\cap H_n$ 
when the orthogonal projections $P_i$ onto $H_i$, $i=1,2,...,n$ are assumed to be known.
The method plays an important role in many areas of mathematics, see, e.g., \cite{Deu92}. 

Define the sequence 
\begin{equation*}
x_0:=x, x_1:=P_1 x_0, x_2:=P_2 x_1,...,x_n:=P_n x_{n-1}
\end{equation*}
and after this
\begin{equation*}
x_{n+1}:=P_1 x_n, x_{n+2}:=P_2 x_{n+1},...,x_{2n}:=P_n x_{2n-1},
\end{equation*}
and so on.
Back in 1962 Halperin \cite{Hal62} proved that $x_k\to P_0 x$ as $k\to\infty$,
where $P_0$ is the orthogonal projection onto the intersection $H_1\cap H_2\cap...\cap H_n$.
A simple and elegant proof of the result can be found in \cite{NS06}.
In particular, the subsequence $x_{nk}=(P_n...P_2 P_1)^k x\to P_0 x$ as $k\to\infty$.
What can be said about the rate of convergence of $\{x_{nk}|k\geqslant 1\}$ to $P_0 x$?
To answer this question Badea, Grivaux and M\"{u}ller 
in \cite{BGM10}, \cite{BGM11} introduced the Friedrichs number of $n$ subspaces,
$c_F(H_1,...,H_n)$.

\subsection{The Friedrichs number of $n$ subspaces}

Badea, Grivaux and M\"{u}ller noticed that for two subspaces $H_1,H_2$
\begin{align*}
c_F(H_1,H_2)&=\sup\{\dfrac{2 Re\langle x_1,x_2\rangle}{\|x_1\|^2+\|x_2\|^2}\,|\\
&x_1\in H_1\ominus (H_1\cap H_2), x_2\in H_2\ominus (H_1\cap H_2), (x_1,x_2)\neq (0,0)\}=\\
&=\sup\{\dfrac{\langle x_1,x_2\rangle+\langle x_2,x_1\rangle}{\|x_1\|^2+\|x_2\|^2}\,|\\
&x_1\in H_1\ominus (H_1\cap H_2), x_2\in H_2\ominus (H_1\cap H_2), (x_1,x_2)\neq (0,0)\} 
\end{align*}
and defined
\begin{align*}
c_F(H_1,...,H_n)&:=\sup\{\dfrac{2}{n-1}\dfrac{\sum_{i<j}Re\langle x_i,x_j\rangle}{\|x_1\|^2+\|x_2\|^2+...+\|x_n\|^2}\,|\\
&x_i\in H_i\ominus (H_1\cap H_2\cap...\cap H_n), i=1,2,...,n, (x_1,x_2,...,x_n)\neq (0,0,...,0)\}=\\
&=\sup\{\dfrac{1}{n-1}\dfrac{\sum_{i\neq j}\langle x_i,x_j\rangle}{\|x_1\|^2+\|x_2\|^2+...+\|x_n\|^2}\,|\\
&x_i\in H_i\ominus (H_1\cap H_2\cap...\cap H_n), i=1,2,...,n, (x_1,x_2,...,x_n)\neq (0,0,...,0)\}.
\end{align*}

Since this definition seems to be rather difficult, we will present a more simple formula for $c_F$.
But first we define the Dixmier number of $n$ subspaces, $c_D(H_1,...,H_n)$.
Following \cite{BGM11}, set
\begin{align*}
c_D(H_1,...,H_n)&:=\sup\{\dfrac{2}{n-1}\dfrac{\sum_{i<j}Re\langle x_i,x_j\rangle}{\|x_1\|^2+\|x_2\|^2+...+\|x_n\|^2}\,|\\
&x_i\in H_i, i=1,2,...,n, (x_1,x_2,...,x_n)\neq (0,0,...,0)\}=\\
&=\sup\{\dfrac{1}{n-1}\dfrac{\sum_{i\neq j}\langle x_i,x_j\rangle}{\|x_1\|^2+\|x_2\|^2+...+\|x_n\|^2}\,|\\
&x_i\in H_i, i=1,2,...,n, (x_1,x_2,...,x_n)\neq (0,0,...,0)\}.
\end{align*}
It is clear that
$$
c_F(H_1,...,H_n)=c_D(H_1\ominus H_0,...,H_n\ominus H_0),
$$
where $H_0=H_1\cap H_2\cap...\cap H_n$.

The Dixmier number of $n$ subspaces is closely related to the sum of the corresponding orthogonal projections.

\begin{proposition}\label{P:c_D and P_1+...+P_n}
The following equality holds:
$$
\|P_1+...+P_n\|=1+(n-1)c_D(H_1,...,H_n).
$$
\end{proposition}

As a corollary, we see that
$$
c_D(H_1,...,H_n)=\dfrac{1}{n-1}\|P_1+...+P_n\|-\dfrac{1}{n-1}
$$
and consequently
$$
c_F(H_1,...,H_n)=c_D(H_1\ominus H_0,...,H_n\ominus H_0)=
\dfrac{1}{n-1}\|P_1+...+P_n-nP_0\|-\dfrac{1}{n-1}.
$$
This equality is not new, see \cite[Proposition 3.7]{BGM11}.

\subsection{The rate of convergence in the method of cyclic alternating projections}

Let us return to the question on the rate of convergence in the method of cyclic alternating projections.
In \cite{BGM11} Badea, Grivaux and M\"{u}ller showed that
\begin{enumerate}
\item
if $c_F(H_1,...,H_n)<1$, i.e., if the angle between $H_1,...,H_n$ is positive, then
\begin{equation*}
\|(P_n...P_2 P_1)^k-P_0\|\leqslant q^k,\,k\geqslant 1
\end{equation*}
for some $q=q(c_F(H_1,...,H_n))\in[0,1)$.
The inequality means that the sequence of operators $(P_n...P_2 P_1)^k$ converges ``quickly'' to $P_0$ as $k\to\infty$.
\item
if $c_F(H_1,...,H_n)=1$, i.e., if the angle between $H_1,...,H_n$ equals zero, then
\begin{equation*}
\|(P_n...P_2 P_1)^k-P_0\|=1,\,k\geqslant 1.
\end{equation*}
Moreover, the sequence of operators $(P_n...P_2 P_1)^k$ converges strongly to $P_0$ as $k\to\infty$
and we have ``arbitrarily slow'' convergence of $(P_n...P_2 P_1)^k$ to $P_0$
(see \cite{BGM11}).
\end{enumerate}
For more complete picture of the quick uniform convergence/arbitrarily slow convergence dichotomy 
see \cite{BGM11} and \cite{BS16}.

\subsection{What this paper is about.}
Let $H$ be a complex Hilbert space and $H_1,...,H_n$ be closed subspaces of $H$.
Denote by $P_i$ the orthogonal projection onto $H_i$, $i=1,...,n$.
Set $H_0:=H_1\cap H_2\cap...\cap H_n$.
Denote by $P_0$ the orthogonal projection onto $H_0$.
This paper is devoted to the study of functions $f_n:[0,1]\to\mathbb{R}$, $n\geqslant 2$,
defined by
\begin{equation*}
f_n(c):=\sup\{\|P_n...P_2 P_1-P_0\|\,|c_F(H_1,...,H_n)\leqslant c\},\, c\in[0,1].
\end{equation*}
The supremum is taken over all systems of subspaces $H_1,...,H_n$ with
$c_F(H_1,...,H_n)\leqslant c$, where $c\in[0,1]$ is a given number.

\begin{remark}
The reader may wonder why we do not write $c_F(H_1,...,H_n)=c$.
Answer: we believe that the assumption $c_F(H_1,...,H_n)\leqslant c$ is more convenient for applications.
Indeed, finding the exact value of $c_F(H_1,...,H_n)$ is usually much more difficult than obtaining the
inequality $c_F(H_1,...,H_n)\leqslant c$.
\end{remark}

\subsection{An equivalent problem}

Let us present a problem which is equivalent to the problem of finding $f_n(c)$.
The fact that these problems are equivalent will be used in the sequel.

\begin{proposition}\label{prop:f_n and c_D}
For every $c\in[0,1]$
\begin{equation*}
f_n(c)=\sup\{\|P_n...P_2 P_1\|\,| c_D(H_1,...,H_n)\leqslant c\},
\end{equation*}
where the supremum is taken over all systems of subspaces $H_1,...,H_n$ with
$c_D(H_1,...,H_n)\leqslant c$.
\end{proposition}

Now from Propositions~\ref{prop:f_n and c_D} and~\ref{P:c_D and P_1+...+P_n} it follows that
$$
f_n(c)=\sup\{\|P_n...P_2 P_1\|\,|\,\|P_1+...+P_n\|\leqslant 1+(n-1)c\}.
$$

\subsection{An application of $f_n$}\label{ss:application}

Using the functions $f_n$ one can easily estimate the rate of convergence in the method of cyclic alternating projections.
Indeed, we have
\begin{equation*}
\|(P_n...P_2 P_1)^k x-P_0 x\|=\|((P_n...P_2 P_1)^k-P_0)x\|\leqslant\|(P_n...P_2 P_1)^k-P_0\|\|x\|.
\end{equation*}
With respect to the orthogonal decomposition $H=H_0\oplus(H\ominus H_0)$ we have
$P_i=I\oplus P_i'$, $i=1,2,...,n$ and $P_0=I\oplus 0$.
Hence 
\begin{equation*}
(P_n...P_2 P_1)^k-P_0=0\oplus (P_n'...P_2'P_1')^k=(P_n...P_2 P_1-P_0)^k
\end{equation*} 
and
\begin{equation*}
\|(P_n...P_2 P_1)^k x-P_0 x\|\leqslant\|P_n...P_2P_1-P_0\|^k\|x\|\leqslant (f_n(c))^k\|x\|,
\end{equation*}
where $c_F(H_1,...,H_n)\leqslant c$.

\subsection{Notation}
Throughout this paper $H$ is a complex Hilbert space.
The inner product in $H$ is denoted by $\langle\cdot,\cdot\rangle$ and $\|\cdot\|$
stands for the corresponding norm, $\|x\|=\sqrt{\langle x,x\rangle}$.
The identity operator on $H$ is denoted by $I$
(throughout the paper it is clear which Hilbert space is being considered).
All vectors are vector-columns; the letter "t" means transpose.

\section{Results and Questions}

Our \textbf{Main Problem} is the following: find $f_n(c), c\in[0,1]$ for $n\geqslant 2$.
It is trivial that $f_2(c)=c, c\in[0,1]$
(this follows from the equality $\|P_1 P_2-P_0\|=c_F(H_1,H_2)$).
But what about $f_n$, $n\geqslant 3$?
Or, at least, what about $f_3$?

\subsection{The functions $f_n$ and an optimization problem}

We will show that our Main Problem is equivalent to 
a certain optimization problem on a subset of the set of
Hermitian complex $n\times n$ matrices.
For two Hermitian $n\times n$ matrices $A,B$
we will write $A\leqslant B$ if
$\langle Ax,x\rangle\leqslant\langle Bx,x\rangle$ for every
$x\in\mathbb{C}^n$, where $\langle\cdot,\cdot\rangle$ is the standard
inner product in the space $\mathbb{C}^n$.
Equivalently, $A\leqslant B$ if the matrix $B-A$ is positive semidefinite.

\begin{theorem}\label{Th:sup over matrices}
The following equality holds:
\begin{equation*}
f_n(c)=\max\{|a_{12}a_{23}...a_{n-1,n}|\}
\end{equation*}
where the maximum is taken over all Hermitian 
complex matrices $A=(a_{ij}|i,j=1,...,n)$
such that $a_{ii}=1$, $i=1,...,n$ and $0\leqslant A\leqslant (1+(n-1)c)I$.
\end{theorem}

Now it's time for some notation.
For an $n\times n$ matrix $A$ set 
$$
\Pi(A):=|a_{12}a_{23}...a_{n-1,n}|.
$$
For a real number $t\geqslant 1$ denote by $\mathcal{H}_n(t)$ the set of all
Hermitian matrices $A=(a_{ij}|i,j=1,...,n)$ such that
$a_{ii}=1$ for $i=1,2,...,n$ and $0\leqslant A\leqslant tI$.
Then Theorem~\ref{Th:sup over matrices} says that
$$
f_n(c)=\max\{\Pi(A)|A\in\mathcal{H}_n(1+(n-1)c)\}.
$$
The following natural problem arises.\\
\textbf{Problem 1}:
find an optimal matrix $A$ for the optimization problem above, i.e.,
a matrix $A\in\mathcal{H}_n(1+(n-1)c)$ such that $\Pi(A)=f_n(c)$.
Or, at least, find 1-diagonal $(a_{12},a_{23},...,a_{n-1,n})$ of an optimal matrix.

It is natural to try to reduce the set of matrices 
on which the function $\Pi$ is considered.
To this end we will use the following lemma.

\begin{lemma}\label{L:lemma good matrix}
Let $t\geqslant 1$.
For arbitrary matrix $A\in\mathcal{H}_n(t)$ 
there exists a matrix $B\in\mathcal{H}_n(t)$ such that
\begin{enumerate}
\item
$b_{ij}\in\mathbb{R}$ for all $i,j=1,2,...,n$ and 
$b_{i,i+1}\geqslant 0$ for $i=1,2,...,n-1$;
\item
$b_{i,j}=b_{n+1-i,n+1-j}$ for all $i,j=1,2,...,n$;
\item
$\Pi(B)\geqslant\Pi(A)$.
\end{enumerate}
\end{lemma}

Denote by $\mathcal{H}_n'(t)$ the set of all matrices $A\in\mathcal{H}_n(t)$
such that
\begin{enumerate}
\item
$a_{ij}\in\mathbb{R}$ for all $i,j=1,...,n$ and 
$a_{i,i+1}\geqslant 0$ for all $i=1,2,...,n-1$;
\item
$a_{ij}=a_{n+1-i,n+1-j}$ for all $i,j=1,2,...,n$.
\end{enumerate}
For example, matrices from $\mathcal{H}_3'(t)$ have the form
\begin{equation*}
A=\begin{pmatrix}
1&x&y\\
x&1&x\\
y&x&1
\end{pmatrix},
\end{equation*}
where $x\geqslant 0$ and $y\in\mathbb{R}$;
matrices from $\mathcal{H}_4'(t)$ have the form
\begin{equation*}
\begin{pmatrix}
1&x&z&w\\
x&1&y&z\\
z&y&1&x\\
w&z&x&1
\end{pmatrix},
\end{equation*}
where $x\geqslant 0$, $y\geqslant 0$, $w,z\in\mathbb{R}$.

Using Lemma~\ref{L:lemma good matrix}, we can write
$$
f_n(c)=\max\{\Pi(A)|A\in\mathcal{H}_n'(1+(n-1)c)\}.
$$
Now it is natural to specify Problem 1.\\
\textbf{Problem 1$'$}: find an optimal matrix $A\in\mathcal{H}_n'(1+(n-1)c)$
for the optimization problem above.
Or, at least, find a 1-diagonal $(a_{12},a_{23},...,a_{n-1,n})$ 
of an optimal matrix $A\in\mathcal{H}_n'(1+(n-1)c)$.

\begin{remark}
It is worth mentioning that there exists a unique optimal 1-diagonal\\
$(a_{12},a_{23},...,a_{n-1,n})$
(i.e., 1-diagonal of an optimal matrix) for which $a_{i,i+1}\geqslant 0$, $i=1,2,...,n-1$.
Indeed, assume that $A,B\in\mathcal{H}_n(1+(n-1)c)$ are optimal 
and $a_{i,i+1}\geqslant 0$, $b_{i,i+1}\geqslant 0$ for $i=1,2,...,n-1$.
We claim that $a_{i,i+1}=b_{i,i+1}$, $i=1,2,...,n-1$.

If $c=0$, then $\mathcal{H}_n(1)=\{I\}$ and our assertion is clear.
Assume that $c>0$.
Then $a_{i,i+1}>0$ and $b_{i,i+1}>0$ for $i=1,2,...,n-1$.
Consider the matrix $C:=1/2(A+B)$.
It is clear that $C\in\mathcal{H}_n(1+(n-1)c)$.
Since $c_{i,i+1}=1/2(a_{i,i+1}+b_{i,i+1})\geqslant\sqrt{a_{i,i+1}b_{i,i+1}}$, we conclude that
$\Pi(C)\geqslant\sqrt{\Pi(A)\Pi(B)}=f_n(c)$.
If follows that $\Pi(C)=f_n(c)$ and consequently $a_{i,i+1}=b_{i,i+1}$ for all $i=1,2,...,n-1$.
\end{remark}

\subsection{The function $f_3$}

Now we are ready to find $f_3$.

\begin{theorem}\label{Th:f3}
We have
$$
f_3(c)=\begin{cases}
4c^2,\quad\text{if}\quad c\in[0,1/4],\\
c,\quad\text{if}\quad c\in[1/4,1].
\end{cases}
$$
For $c\in[0,1/4]$ the matrix
\begin{equation*}
\begin{pmatrix}
1   &  2c  &-2c\\
2c  &  1   & 2c\\
-2c &  2c  & 1
\end{pmatrix}
\end{equation*}
is optimal, for $c\in[1/4,1]$ the matrix
\begin{equation*}
\begin{pmatrix}
1        &\sqrt{c}  &  2c-1\\
\sqrt{c} &1         &\sqrt{c}\\
2c-1     &\sqrt{c}  & 1
\end{pmatrix}
\end{equation*}
is optimal.
\end{theorem}

\subsection{On the functions $f_n$ with $n\geqslant 4$}

For $n=4$, to find $f_4(c)$ one have to consider matrices of the form
\begin{equation*}
\begin{pmatrix}
1&x&z&w\\
x&1&y&z\\
z&y&1&x\\
w&z&x&1
\end{pmatrix},
\end{equation*}
where $x\geqslant 0$, $y\geqslant 0$, $w,z\in\mathbb{R}$,
and have to maximize $x^2 y$.
We could not find $f_4$ (and $f_n$ for $n\geqslant 4$).
Nevertheless we have the following theorem.

\begin{theorem}\label{Th:f_n}
Let $n\geqslant 2$ and $c\in[0,1/(n-1)^2]$.
Then $f_n(c)=(n-1)^{n-1}c^{n-1}$ and the matrix $A\in\mathcal{H}_n'(1+(n-1)c)$ defined by
$$a_{ij}=\begin{cases}
1,\quad\text{if}\quad i=j,\\
(-1)^{i+j+1}(n-1)c,\quad\text{if}\quad i\neq j
\end{cases}
$$
is optimal.
\end{theorem}

Although we could not find $f_n$ for $n\geqslant 4$,
we know some properties of the function.
Firstly, note that $f_n$ is non-decreasing on $[0,1]$
(it follows directly from the definition of $f_n$).

\begin{theorem}\label{Th:convexity_fn}
The function $f_n^{1/(n-1)}$ is concave on $[0,1]$.
\end{theorem}

\begin{corollary}\label{C:continuity of f_n}
The function $f_n$ is continuous on $[0,1]$.
\end{corollary}

\begin{theorem}\label{Th:functional equation}
The function $f_n$ satisfies the following functional equation:
\begin{equation*}
f_n\left(\dfrac{1}{(n-1)^2 c}\right)=\dfrac{f_n(c)}{(n-1)^{n-1}c^{n-1}}
\end{equation*}
for $c\in[1/(n-1)^2,1]$.
\end{theorem}

Regarding Problem 1$'$, we have the following criterion for a matrix to be optimal.

\begin{proposition}\label{P:optimal matrix}
Let $c>0$ and a matrix $A\in\mathcal{H}_n(1+(n-1)c)$ 
be such that $a_{i,i+1}>0$ for $i=1,2,...,n$.
Then $A$ is optimal, i.e., $\Pi(A)=f_n(c)$ if and only if
$$
\dfrac{b_{12}}{a_{12}}+\dfrac{b_{23}}{a_{23}}+...+\dfrac{b_{n-1,n}}{a_{n-1,n}}\leqslant n-1
$$
for arbitrary matrix $B\in\mathcal{H}_n(1+(n-1)c)$ 
with $b_{i,i+1}\geqslant 0$ for $i=1,2,...,n-1$.
\end{proposition}

\subsection{Bounds for $f_n(c)$: known upper bounds}

In this subsection we present known upper bounds for $f_n(c)$.
The upper bounds are of great interest because using them
one can easily estimate the rate of convergence in the method of cyclic alternating projections
(see subsection~\ref{ss:application}).

Let $H$ be a Hilbert space and $H_1,...,H_n$ be closed subspaces of $H$.
Denote by $P_i$ the orthogonal projection onto $H_i$, $i=1,...,n$.
Set $H_0:=H_1\cap H_2\cap...\cap H_n$.
Denote by $P_0$ the orthogonal projection onto $H_0$.
Set $c_F:=c_F(H_1,...,H_n)$.
Let $c\in[0,1]$.
In what follows we assume that $c_F\leqslant c$.

Badea, Grivaux and M\"{u}ller \cite{BGM11} showed that
\begin{equation*}
\|P_n...P_2 P_1-P_0\|\leqslant\sqrt{1-\dfrac{(1-c_F)^2}{16n^2}}.
\end{equation*}
It follows that
\begin{equation*}
f_n(c)\leqslant\sqrt{1-\dfrac{(1-c)^2}{16n^2}}.
\end{equation*}

Badea and Seifert \cite{BS16} showed that
\begin{equation*}
\|P_n...P_2 P_1-P_0\|\leqslant\sqrt{1-\dfrac{3(n-1)}{n^3}(1-c_F)}.
\end{equation*}
It follows that
\begin{equation*}
f_n(c)\leqslant\sqrt{1-\dfrac{3(n-1)}{n^3}(1-c)}.
\end{equation*}

\subsection{Our upper bound for $f_n(c)$}

\begin{theorem}\label{Th:inequality for fn}
The following inequality holds:
\begin{equation*}
f_n(c)\leqslant
\sqrt{\dfrac{n-4(n-1)(\sin^2(\pi/(2n)))(1-c)}{n+4(n-1)^2(\sin^2(\pi/(2n)))(1-c)}}
\end{equation*}
for every $c\in[0,1]$.
\end{theorem}

By using Theorem~\ref{Th:inequality for fn} one can get a more simple estimate for $f_n(c)$
(a more simple than the estimate given by Theorem~\ref{Th:inequality for fn}).
One can easily check that $a/b\leqslant(a+x)/(b+x)$, 
where $0\leqslant a\leqslant b$ and $x\geqslant 0$.
Setting $a=n-4(n-1)(\sin^2(\pi/(2n)))(1-c)$,
$b=n+4(n-1)^2(\sin^2(\pi/(2n)))(1-c)$ and
$x=4(n-1)(\sin^2(\pi/(2n)))(1-c)$, we get
\begin{align*}
&f_n(c)\leqslant 
\sqrt{\dfrac{n-4(n-1)(\sin^2(\pi/(2n)))(1-c)}{n+4(n-1)^2(\sin^2(\pi/(2n)))(1-c)}}\leqslant\\
&\leqslant\sqrt{\dfrac{n}{n+4n(n-1)(\sin^2(\pi/(2n)))(1-c)}}=
\dfrac{1}{\sqrt{1+4(n-1)(\sin^2(\pi/(2n)))(1-c)}}.
\end{align*}
Using Taylor's theorem with the Lagrange form of the remainder one can easily check that
\begin{equation*}
\dfrac{1}{\sqrt{1+u}}\leqslant 1-\dfrac{1}{2}u+\dfrac{3}{8}u^2
\end{equation*}
for $u\geqslant 0$.
Thus
\begin{equation*}
f_n(c)\leqslant\dfrac{1}{\sqrt{1+4(n-1)(\sin^2(\pi/(2n)))(1-c)}} 
\leqslant 1-a_n(1-c)+b_n(1-c)^2,
\end{equation*}
where $a_n=2(n-1)\sin^2(\pi/(2n))$ and $b_n=6(n-1)^2\sin^4(\pi/(2n))$.
Note that $a_2=1$ and $a_3=1$.

\textbf{Question 1.}
Is it true that $f_n(c)\leqslant 1-a_n(1-c)$ for all $c\in[0,1]$?
Or, at least, for all $c$ which are sufficiently close to $1$?

\subsection{Bounds for $f_n$: lower bounds}

\begin{theorem}\label{Th:lower bound}
For every $n\geqslant 2$ there exists a positive constant $\widetilde{b}_n$ such that
$$
f_n(c)\geqslant 1-a_n(1-c)-\widetilde{b}_n(1-c)^2
$$
for all $c\in[0,1]$.
\end{theorem}

Consequently, we have
$$
1-a_n(1-c)-\widetilde{b}_n(1-c)^2\leqslant f_n(c)\leqslant 1-a_n(1-c)+b_n(1-c)^2
$$
for all $c\in[0,1]$.
These inequalities mean that the estimate for $f_n(c)$ 
given by Theorem~\ref{Th:inequality for fn} is optimal for $c\approx 1$,
up to $O((1-c)^2)$, $c\to 1-$.

\section{Proofs}

\subsection{Proof of Proposition~\ref{P:c_D and P_1+...+P_n}}

Set $c=c_D(H_1,...,H_n)$.
Since
$$
c=\sup\{\dfrac{1}{n-1}\dfrac{\sum_{i\neq j}\langle x_i,x_j\rangle}{\sum_{i=1}^n \|x_i\|^2}\,|
x_i\in H_i, i=1,...,n, (x_1,...,x_n)\neq (0,...,0)\},
$$
we conclude that
\begin{align*}
1+(n-1)c&=\sup\{\dfrac{\sum_{i,j}\langle x_i,x_j\rangle}{\sum_{i=1}^n \|x_i\|^2}\,|
x_i\in H_i, i=1,...,n, (x_1,...,x_n)\neq (0,...,0)\}=\\
&=\sup\{\dfrac{\|x_1+...+x_n\|^2}{\|x_1\|^2+...+\|x_n\|^2}\,|
x_i\in H_i, i=1,...,n, (x_1,...,x_n)\neq (0,...,0)\}.
\end{align*}
Consider an operator $S:H_1\oplus H_2\oplus...\oplus H_n\to H$ defined by
$$
S(x_1,...,x_n)^t=x_1+...+x_n,\,x_i\in H_i,\,i=1,2,...,n.
$$
Then $1+(n-1)c=\|S\|^2$.
It is easy to check that $S^*:H\to H_1\oplus...\oplus H_n$ acts as follows:
$S^*x=(P_1 x,...,P_n x)^t$, $x\in H$.
Thus $SS^*=P_1+...+P_n$ and
$$
\|P_1+...+P_n\|=\|SS^*\|=\|S^*\|^2=\|S\|^2=1+(n-1)c.
$$
The proof is complete.

\subsection{Proof of Proposition~\ref{prop:f_n and c_D}}

Define a function $g_n:[0,1]\to\mathbb{R}$ by
$$
g_n(c)=\sup\{\|P_n...P_2 P_1\|\,|c_D(H_1,...,H_n)\leqslant c\},\,c\in[0,1].
$$
We have to prove that $f_n(c)=g_n(c)$ for every $c\in[0,1]$.

First, we will show that $f_n(c)\leqslant g_n(c)$, $c\in[0,1]$.
Consider arbitrary system of subspaces $H_1,...,H_n$ of a Hilbert space $H$ such that
$c_F(H_1,...,H_n)\leqslant c$.
Set $H_0:=H_1\cap...\cap H_n$ and denote by $P_0$ the orthogonal projection onto $H_0$.
Let us prove that $\|P_n...P_2 P_1-P_0\|\leqslant g_n(c)$.
To this end consider the orthogonal decomposition $H=H_0\oplus (H\ominus H_0)=:H_0\oplus H'$.
With respect to this orthogonal decomposition $H_i=H_0\oplus (H_i\ominus H_0)=:H_0\oplus H_i'$,
$i=1,2,...,n$.
Thus
\begin{align*}
c_F(H_1,...,H_n)&:=\sup\{\dfrac{2}{n-1}\dfrac{\sum_{i<j}Re\langle x_i,x_j\rangle}{\|x_1\|^2+\|x_2\|^2+...+\|x_n\|^2}\,|\\
&x_i\in H_i\ominus H_0, i=1,2,...,n, (x_1,x_2,...,x_n)\neq (0,0,...,0)\}=\\
&=\sup\{\dfrac{2}{n-1}\dfrac{\sum_{i<j}Re\langle x_i,x_j\rangle}{\|x_1\|^2+\|x_2\|^2+...+\|x_n\|^2}\,|\\
&x_i\in H_i', i=1,2,...,n, (x_1,x_2,...,x_n)\neq (0,0,...,0)\}=\\
&=c_D(H_1',...,H_n').
\end{align*}
Therefore $c_D(H_1',...,H_n')\leqslant c$.
Further, with respect to the orthogonal decomposition $H=H_0\oplus H'$ we have
$P_i=I\oplus P_i'$, where $P_i'$ is the orthogonal projection onto $H_i'$ in $H'$, 
$i=1,2,...,n$, and $P_0=I\oplus 0$.
Thus $P_n...P_2 P_1-P_0=0\oplus P_n'...P_2' P_1'$ whence
$$
\|P_n...P_2 P_1-P_0\|=\|P_n'...P_2'P_1'\|\leqslant g_n(c),
$$
because $c_D(H_1',...,H_n')\leqslant c$.
It follows that $f_n(c)\leqslant g_n(c)$.

Now we will show that $g_n(c)\leqslant f_n(c)$.
Let us prove this inequality for $c\in[0,1)$.
Consider arbitrary system of subspaces $H_1,...,H_n$ of a Hilbert space $H$ such that
$c_D(H_1,...,H_n)\leqslant c$.
Let us prove that $\|P_n...P_2 P_1\|\leqslant f_n(c)$.
Since $c_D(H_1,...,H_n)<1$, we conclude that $H_1\cap...\cap H_n=\{0\}$.
Indeed, assume that $H_1\cap...\cap H_n\neq\{0\}$.
Take a vector $u\in H_1\cap...\cap H_n$, $u\neq 0$ and set $x_i=u$, $i=1,2,...,n$.
Then
$$
\dfrac{1}{n-1}\dfrac{\sum_{i\neq j}\langle x_i,x_j\rangle}{\|x_1\|^2+\|x_2\|^2+...+\|x_n\|^2}=
\dfrac{n(n-1)\|u\|^2}{n(n-1)\|u\|^2}=1
$$
whence $c_D(H_1,...,H_n)=1$, contradiction.
Therefore $H_1\cap...\cap H_n=\{0\}$.
Thus $c_F(H_1,...,H_n)=c_D(H_1,...,H_n)\leqslant c$ and $P_0=0$.
Hence
$$
\|P_n...P_2 P_1\|=\|P_n...P_2 P_1-P_0\|\leqslant f_n(c).
$$
It follows that $g_n(c)\leqslant f_n(c)$.

Let us show that $g_n(1)\leqslant f_n(1)$.
It is clear that
$$
g_n(1)=\sup\{\|P_n...P_2 P_1\|\,|c_D(H_1,...,H_n)\leqslant 1\}=1
$$
(just take $H_i=H$, $i=1,2,...,n$, then $\|P_n...P_2 P_1\|=\|I\|=1$).
So we have to show that $f_n(1)\geqslant 1$.
To this end we will show that the number $\|P_n...P_2 P_1-P_0\|$ can be arbitrarily close to $1$.
Let $H=\mathbb{C}^2$ be the two-dimensional Hilbert space.
For an angle $\varphi\in(0,\pi/2]$ define two subspaces
$$
M=\{(\cos\varphi,\sin\varphi)^t x|x\in\mathbb{C}\}
$$
and
$$
N=\{(1,0)^t x|x\in\mathbb{C}\}=\{(x,0)^t|x\in\mathbb{C}\}=\mathbb{C}\oplus\{0\}.
$$
Then $M\cap N=\{0\}$ and for the orthogonal projections $P_M$ and $P_N$ onto the subspaces $M$ and $N$, respectively,
we have $\|P_N P_M\|=\|P_N(\cos\varphi,\sin\varphi)^t\|=\cos\varphi$.
Thus for a system of $n$ subspaces $H_1=M$, $H_i=N$, $i=2,3,...,n$ we have
$$
\|P_n...P_2 P_1-P_0\|=\|P_N P_M-0\|=\cos\varphi
$$
can be arbitrarily close to $1$.
Therefore $f_n(1)\geqslant 1$.

So, we proved that $f_n(c)\leqslant g_n(c)$ and $g_n(c)\leqslant f_n(c)$.
It follows that $f_n(c)=g_n(c)$, $c\in[0,1]$.

\subsection{Proof of Theorem~\ref{Th:sup over matrices}}

First, note the maximum
$\max\{|a_{12}a_{23}...a_{n-1,n}|\,| A\in\mathcal{H}_n(1+(n-1)c)\}$ exists, i.e., is attained.
This is a direct consequence of the following two facts:
the function $A\mapsto |a_{12}a_{23}...a_{n-1,n}|$ is continuous and the set $\mathcal{H}_n(1+(n-1)c)$
is compact.

Let us show that
\begin{equation*}
f_n(c)\geqslant\max\{|a_{12}a_{23}...a_{n-1,n}|\,| A\in\mathcal{H}_n(1+(n-1)c)\}.
\end{equation*}
To this end we consider arbitrary matrix $A\in\mathcal{H}_n(1+(n-1)c)$.
We have to show that $f_n(c)\geqslant|a_{12}a_{23}...a_{n-1,n}|$.
Since $A$ is Hermitian and positive semidefinite, 
we conclude that $A=B^*B$ for some $n\times n$ matrix $B$.
Let $v_1,...,v_n$ be the columns of $B$, i.e., $B=(v_1 v_2 ... v_n)$.
We have $a_{ij}=\sum_{k=1}^n \overline{b_{ki}}b_{kj}=\langle v_j,v_i\rangle$.
This means that $A$ is the Gram matrix of the vectors $v_1,...,v_n$.
Since $a_{ii}=1$, we see that $\|v_i\|=1$, $i=1,2,...,n$.
Consider the system of one dimensional subspaces $H_i=\{av_i|a\in\mathbb{C}\}$, $i=1,2,...,n$.
We claim that $c_D(H_1,...,H_n)\leqslant c$ and $\|P_n...P_1\|=|a_{12}a_{23}...a_{n-1,n}|$.
It will follow that $f_n(c)\geqslant|a_{12}a_{23}...a_{n-1,n}|$.
First consider
\begin{equation*}
\|P_nP_{n-1}...P_1\|=\|P_n P_{n-1}...P_1|_{H_1}\|=\|P_n P_{n-1}... P_1v_1\|.
\end{equation*}
Since $P_i x=\langle x,v_i\rangle v_i$, $x\in\mathbb{C}^n$, one can easily check that
\begin{equation*}
P_n P_{n-1}...P_1 v_1=\langle v_1,v_2\rangle \langle v_2,v_3\rangle ...\langle v_{n-1},v_n\rangle v_n.
\end{equation*}
It follows that
\begin{equation*}
\|P_n P_{n-1}...P_1 v_1\|=|\langle v_1,v_2\rangle \langle v_2,v_3\rangle ...\langle v_{n-1},v_n\rangle|=
|a_{21}a_{32}...a_{n,n-1}|=|a_{12}a_{23}...a_{n-1,n}|.
\end{equation*} 
Thus $\|P_n...P_1\|=|a_{12}a_{23}...a_{n-1,n}|$.

Let us show that $c_D(H_1,...,H_n)\leqslant c$.
For arbitrary vectors $x_1=a_1 v_1,...,x_n=a_n v_n$ we have
\begin{align*}
&\sum_{i,j=1}^n\langle x_j,x_i\rangle=
\sum_{i,j=1}^n \langle a_j v_j,a_i v_i\rangle=
\sum_{i,j=1}^n a_{ij}a_j \overline{a_i}=
\sum_{i=1}^n\left(\sum_{j=1}^n a_{ij}a_j\right)\overline{a_i}=\\
&=\langle A(a_1,...,a_n)^t,(a_1,...,a_n)^t\rangle\leqslant (1+(n-1)c)\|(a_1,...,a_n)^t\|^2= (1+(n-1)c)\sum_{i=1}^n\|x_i\|^2.
\end{align*}
It follows that $\sum_{i\neq j}\langle x_j,x_i\rangle\leqslant (n-1)c\sum_{i=1}^n\|x_i\|^2$.
Therefore
\begin{equation*}
c_D(H_1,...,H_n)=\sup\{\dfrac{1}{n-1}\dfrac{\sum_{i\neq j}\langle x_j,x_i\rangle}{\sum_{i=1}^n\|x_i\|^2}|
(x_1,...,x_n)\in H_1\times...\times H_n\setminus\{(0,...,0)\}\}\leqslant c.
\end{equation*}

Let us show that
\begin{equation*}
f_n(c)\leqslant\max\{|a_{12}a_{23}...a_{n-1,n}|| A\in\mathcal{H}_n(1+(n-1)c)\}.
\end{equation*}
Define $K:=\max\{|a_{12}a_{23}...a_{n-1,n}|| A\in\mathcal{H}_n(1+(n-1)c)\}$
and consider arbitrary system of subspaces $H_1,...,H_n$ of a Hilbert space $H$ such that $c_D(H_1,...,H_n)\leqslant c$.
We have to prove that $\|P_n... P_2 P_1\|\leqslant K$.
Let $v_1\in H_1,...,v_n\in H_n$ be arbitrary elements with $\|v_i\|=1$, $i=1,...,n$.
Denote by $G$ the Gram matrix of these elements, i.e., $G=(g_{ij}=\langle v_j,v_i\rangle|i,j=1,...,n)$.
We claim that $G\in\mathcal{H}_n(1+(n-1)c)$.
Indeed, it is clear that $G^*=G\geqslant 0$ and $g_{ii}=\|v_i\|^2=1$, $i=1,...,n$.
It remains to show that $G\leqslant (1+(n-1)c)I$.
For arbitrary scalars $a_1,...,a_n$ we have
\begin{align*}
&\langle G(a_1,...,a_n)^t,(a_1,...,a_n)^t\rangle=\sum_{i,j=1}^n g_{ij}a_j\overline{a_i}=\\
&=\sum_{i=1}^n|a_i|^2+\sum_{i\neq j}\langle v_j,v_i\rangle a_j\overline{a_i}=
\sum_{i=1}^n|a_i|^2+\sum_{i\neq j}\langle a_j v_j,a_i v_i\rangle\leqslant\\
&\leqslant \sum_{i=1}^n|a_i|^2+(n-1)c\sum_{i=1}^n\|a_i v_i\|^2=(1+(n-1)c)\sum_{i=1}^n|a_i|^2.
\end{align*}
It follows that $G\leqslant (1+(n-1)c)I$.
(It is worth mentioning that this follows also from \cite[Proposition 3.4]{BGM11} formulated for
the nonreduced configuration constant and \cite[Proposition 3.6(f)]{BGM11}.)
Since $G\in\mathcal{H}_n(1+(n-1)c)$, we conclude that $|g_{12}g_{23}...g_{n-1,n}|\leqslant K$, i.e., 
$|\langle v_1,v_2\rangle \langle v_2,v_3\rangle...\langle v_{n-1},v_n\rangle|\leqslant K$.
It follows that for \emph{arbitrary} elements $u_1\in H_1,...,u_n\in H_n$ we have
\begin{equation}\label{ineq:product}
|\langle u_1,u_2\rangle \langle u_2,u_3\rangle...\langle u_{n-1},u_n\rangle|\leqslant
K\|u_1\|\|u_2\|^2...\|u_{n-1}\|^2\|u_n\|.
\end{equation}
Now consider arbitrary $x\in H$ and set $u_i:=P_i P_{i-1}...P_1 x$, $i=1,...,n$.
Then
\begin{equation*}
\langle u_i,u_{i+1}\rangle=\langle u_i, P_{i+1}u_i\rangle=\|P_{i+1}u_i\|^2=\|u_{i+1}\|^2.
\end{equation*}
Thus by \eqref{ineq:product} we get
\begin{equation*}
\|u_2\|^2\|u_3\|^2...\|u_{n-1}\|^2\|u_n\|^2\leqslant K\|u_1\|\|u_2\|^2...\|u_{n-1}\|^2\|u_n\|,
\end{equation*}
that is, $\|u_n\|\leqslant K\|u_1\|$.
Since $u_1=P_1 x$ and $u_n=P_n P_{n-1}...P_1 x$, we see that
\begin{equation*}
\|P_n... P_2 P_1 x\|\leqslant K\|P_1 x\|\leqslant K\|x\|.
\end{equation*}
Therefore $\|P_n...P_2 P_1\|\leqslant K$.
This completes the proof.

\subsection{Proof of Lemma~\ref{L:lemma good matrix}}

First, note that the set $\mathcal{H}_n(t)$ has the following properties:
\begin{enumerate}
\item
if $A\in\mathcal{H}_n(t)$ and $U$ is a diagonal unitary matrix, i.e.,
$U=diag(u_1,...,u_n)$, where $u_1,...,u_n$ are scalars with $|u_i|=1$, $i=1,2,...,n$,
then $U^*AU\in\mathcal{H}_n(t)$;
\item
if $A\in\mathcal{H}_n(t)$, then $A^\top\in\mathcal{H}_n(t)$.
Here $(A^\top)_{ij}=a_{ji}$, $i,j=1,2,...,n$;
\item
if $A\in\mathcal{H}_n(t)$, then $\overleftarrow{A}\in\mathcal{H}_n(t)$.
Here $(\overleftarrow{A})_{ij}=a_{n+1-i,n+1-j}$, $i,j=1,2,...,n$.
\item
the set $\mathcal{H}_n(t)$ is convex.
\end{enumerate}

Now we are ready to prove the needed assertion.
Let $A\in\mathcal{H}_n(t)$.
For a diagonal unitary matrix $U=diag(u_1,u_2,...,u_n)$ define $B:=U^*AU$.
Then $B\in\mathcal{H}_n(t)$.
Moreover, since $b_{i,i+1}=a_{i,i+1}\overline{u_i}u_{i+1}$
one can choose scalars $u_1,...,u_n$ so that $b_{i,i+1}=|a_{i,i+1}|$
for $i=1,2,...,n-1$.
Then $\Pi(B)=\Pi(A)$.

Further, consider the matrix $B^\top$ and set $C:=1/2(B+B^\top)$.
Then $C\in\mathcal{H}_n(t)$.
We have $c_{ij}=1/2(b_{ij}+b_{ji})=Re(b_{ij})\in\mathbb{R}$
and $c_{i,i+1}=b_{i,i+1}\geqslant 0$.
Therefore $\Pi(C)=\Pi(B)$.

Finally, consider the matrix $\overleftarrow{C}$ and set $D:=1/2(C+\overleftarrow{C})$.
Then $D\in\mathcal{H}_n(t)$.
The matrix $D$ has the following properties:
\begin{enumerate}
\item
$d_{ij}=1/2(c_{ij}+c_{n+1-i,n+1-j})\in\mathbb{R}$ for all $i,j$ and 
$d_{i,i+1}=1/2(c_{i,i+1}+c_{n+1-i,n-i})=1/2(c_{i,i+1}+c_{n-i,n-i+1})\geqslant 0$ for $i=1,2,...,n-1$;
\item
$d_{n+1-i,n+1-j}=d_{ij}$ for all $i,j=1,2,...,n$;
\item
since $d_{i,i+1}=1/2(c_{i,i+1}+c_{n-i,n-i+1})\geqslant\sqrt{c_{i,i+1}c_{n-i,n-i+1}}$ for $i=1,2,...,n-1$,
we conclude that 
$$
\Pi(D)=d_{12}d_{23}...d_{n-1,n}\geqslant c_{12}c_{23}...c_{n-1,n}=\Pi(C).
$$
\end{enumerate}
Thus $D$ is a needed matrix.

\subsection{Proof of Theorem~\ref{Th:f3}}

To find $f_3(c)$ one can consider matrices of the form
\begin{equation*}
A=\begin{pmatrix}
1&x&y\\
x&1&x\\
y&x&1
\end{pmatrix},
\end{equation*}
where $x\geqslant 0$ and $y\in\mathbb{R}$.
We have to maximize $x^2$ under the condition 
$0\leqslant A\leqslant(1+2c)I$.

Consider the condition $A\geqslant 0$.
It is well-known that a Hermitian matrix is positive semidefinite if and only if
every principal minor of the matrix (including its determinant) is nonnegative.
(Recall that a principal minor is the determinant of a principal submatrix;
a principal submatrix is a square submatrix obtained by removing certain
rows and columns with the same index sets.)
Using this criterion one can easily check that $A\geqslant 0$ if and only if
\begin{equation*}
\begin{cases}
0\leqslant x\leqslant 1,\\
|y|\leqslant 1,\\
x^2\leqslant (1+y)/2.
\end{cases}
\end{equation*}

Consider the condition $A\leqslant (1+2c)I$
$\Leftrightarrow$ $(1+2c)I-A\geqslant 0$.
Now one can easily check that $A\leqslant (1+2c)I$ if and only if
\begin{equation*}
\begin{cases}
0\leqslant x \leqslant 2c,\\
|y|\leqslant 2c,\\
x^2\leqslant c(2c-y).
\end{cases}
\end{equation*}
Hence, $0\leqslant A\leqslant (1+2c)I$ if and only if
\begin{equation*}
\begin{cases}
0\leqslant x\leqslant\min\{1,2c\},\\
|y|\leqslant\min\{1,2c\},\\
x^2\leqslant\min\{(1+y)/2, c(2c-y)\}.
\end{cases}
\end{equation*}
We have to maximize $x^2$ under these conditions.

Define two linear functions $\varphi(y)=(1+y)/2$ and $\psi(y)=c(2c-y)$.
It is clear that $\varphi$ is increasing and $\psi$ is nonincreasing.
Consider the equation $\varphi(y)=\psi(y)$.
The unique solution is $y=2c-1$.
Therefore
\begin{equation*}
\min\{(1+y)/2, c(2c-y)\}=
\begin{cases}
(1+y)/2,\quad\text{if}\quad y\leqslant 2c-1,\\
c(2c-y),\quad\text{if}\quad y\geqslant 2c-1.
\end{cases}
\end{equation*}
This minimum attains its maximum value $c$ at the point $y=2c-1$.
Thus $x^2\leqslant c$ and $x\leqslant\sqrt{c}$.
Let us check for which $c\in[0,1]$ the values $x=\sqrt{c}$ and $y=2c-1$ are permissible.
First consider the inequality $|y|\leqslant\min\{1,2c\}$.
It is clear that $-1\leqslant 2c-1\leqslant 1$ and $2c-1\leqslant 2c$.
However, the inequality $2c-1\geqslant -2c$ holds only for $c\geqslant 1/4$.
For such $c$ we have $\sqrt{c}\leqslant 1$ and $\sqrt{c}\leqslant 2c$.
Conclusion: for $c\in[1/4,1]$ the optimal values $x=\sqrt{c}$, $y=2c-1$, the optimal matrix is equal to
\begin{equation*}
\begin{pmatrix}
1        &\sqrt{c}  &  2c-1\\
\sqrt{c} &1         &\sqrt{c}\\
2c-1     &\sqrt{c}  & 1
\end{pmatrix}
\end{equation*}
and $f_3(c)=c$.

Consider the case $c\in[0,1/4)$.
Then $2c-1<-2c$ and hence the conditions for $x$ and $y$ can be rewritten as
\begin{equation*}
\begin{cases}
0\leqslant x\leqslant 2c,\\
-2c\leqslant y\leqslant 2c,\\
x^2\leqslant c(2c-y).
\end{cases}
\end{equation*}
Now it is easy to see that the optimal values of $x$ and $y$ are $x=2c$ and $y=-2c$.
Therefore the optimal matrix is equal to
\begin{equation*}
\begin{pmatrix}
1   &  2c  &-2c\\
2c  &  1   & 2c\\
-2c &  2c  & 1
\end{pmatrix}
\end{equation*}
and $f_3(c)=4c^2$.

\subsection{Proof of Theorem~\ref{Th:f_n}}

Consider an arbitrary matrix $A\in\mathcal{H}_n(1+(n-1)c)$.
Since $A\leqslant (1+(n-1)c)I$, we conclude that the matrix $(1+(n-1)c)I-A$
is positive semidefinite.
It follows that the determinant of every $2\times 2$ submatrix
\begin{equation*}
\begin{pmatrix}
(n-1)c & -a_{ij}\\
-a_{ji}& (n-1)c
\end{pmatrix}
\end{equation*}
is nonnegative, i.e., $(n-1)^2 c^2-|a_{ij}|^2\geqslant 0$, $|a_{ij}|\leqslant (n-1)c$.
Therefore $|a_{12}a_{23}...a_{n-1,n}|\leqslant (n-1)^{n-1}c^{n-1}$.

On the other hand, consider the matrix $J$ where each entry is equal to $1$, i.e.,
\begin{equation*}
J=\begin{pmatrix}
1     &\ldots  &1    \\
\vdots&\ddots&\vdots \\
1     &\ldots & 1
\end{pmatrix}.
\end{equation*}
It is easily seen that $J$ is positive semidefinite and the largest eigenvalue of $J$ equals $n$.
Thus $0\leqslant J\leqslant nI$,
$-nI\leqslant -J\leqslant 0$,
$-(n-1)I\leqslant I-J\leqslant I$,
$-(n-1)^2 c I\leqslant (n-1)c(I-J)\leqslant (n-1)c I$ and
\begin{equation*}
(1-(n-1)^2 c)I\leqslant I+(n-1)c(I-J)\leqslant (1+(n-1)c)I.
\end{equation*}
Set $M:=I+(n-1)c(I-J)$.
Since $c\in[0,1/(n-1)^2]$ and
\begin{equation*}
m_{ij}=\begin{cases}
1,\quad\text{if}\quad i=j,\\
-(n-1)c,\quad\text{if}\quad i\neq j,
\end{cases}
\end{equation*}
we see that $M\in\mathcal{H}_n(1+(n-1)c)$ and $\Pi(M)=(n-1)^{n-1}c^{n-1}$.
Therefore $f_n(c)=(n-1)^{n-1}c^{n-1}$.

Finally, define $U:=diag(-1,1,-1,1,...)$ and consider the matrix $A:=U^*MU$.
Since
$$
a_{ij}=\begin{cases}
1,\quad\text{if}\quad i=j,\\
(-1)^{i+j+1}(n-1)c,\quad\text{if}\quad i\neq j,
\end{cases}
$$
we conclude that $A\in\mathcal{H}_n'(1+(n-1)c)$ and $\Pi(A)=(n-1)^{n-1}c^{n-1}$.
Thus $A$ is optimal.

\subsection{Proof of Theorem~\ref{Th:convexity_fn}}

Let $c_1,c_2\in[0,1]$ and $\lambda\in(0,1)$.
We have to show that
$$
(f_n(\lambda c_1+(1-\lambda)c_2))^{1/(n-1)}\geqslant\lambda(f_n(c_1))^{1/(n-1)}+(1-\lambda)(f_n(c_2))^{1/(n-1)}.
$$
Let $A\in\mathcal{H}_n(1+(n-1)c_1)$ be such that $a_{i,i+1}\geqslant 0$ for $i=1,2,...,n-1$ and $\Pi(A)=f_n(c_1)$.
Let $B\in\mathcal{H}_n(1+(n-1)c_2)$ be such that $b_{i,i+1}\geqslant 0$ for $i=1,2,...,n-1$ and $\Pi(B)=f_n(c_2)$.
Consider the matrix $\lambda A+(1-\lambda)B$.
It is clear that $\lambda A+(1-\lambda)B\in\mathcal{H}_n(1+(n-1)(\lambda c_1+(1-\lambda)c_2))$.
Thus
$$
f_n(\lambda c_1+(1-\lambda)c_2)\geqslant\Pi(\lambda A+(1-\lambda)B)
$$
whence
$$
(f_n(\lambda c_1+(1-\lambda)c_2))^{1/(n-1)}\geqslant(\Pi(\lambda A+(1-\lambda)B))^{1/(n-1)}.
$$
Now we will use the inequality
$$
\sqrt[m]{(s_1+t_1)...(s_m+t_m)}\geqslant\sqrt[m]{s_1 s_2... s_m}+\sqrt[m]{t_1 t_2...t_m},
$$
where $m$ is a natural number and numbers $s_1,...,s_m,t_1,...,t_m$ are nonnegative.
We have
\begin{align*}
&(f_n(\lambda c_1+(1-\lambda)c_2))^{1/(n-1)}\geqslant(\Pi(\lambda A+(1-\lambda)B))^{1/(n-1)}=\\
&=\sqrt[n-1]{(\lambda a_{12}+(1-\lambda)b_{12})...(\lambda a_{n-1,n}+(1-\lambda)b_{n-1,n})}\geqslant\\
&\geqslant\sqrt[n-1]{(\lambda a_{12})...(\lambda a_{n-1,n})}+\sqrt[n-1]{((1-\lambda) b_{12})...((1-\lambda) b_{n-1,n})}=\\
&=\lambda(\Pi(A))^{1/(n-1)}+(1-\lambda)(\Pi(B))^{1/(n-1)}=
\lambda(f_n(c_1))^{1/(n-1)}+(1-\lambda)(f_n(c_2))^{1/(n-1)}.
\end{align*}
The proof is completed.

\subsection{Proof of Corollary~\ref{C:continuity of f_n}}

Define the function $g_n:=f_n^{1/(n-1)}$.
Let us prove that $g_n$ is continuous on $[0,1]$.
It will follow that $f_n=g_n^{n-1}$ is also continuous on $[0,1]$.

We will use the following well-known fact:
if a function $\varphi:(a,b)\to\mathbb{R}$ is convex on $(a,b)$,
then $\varphi$ is continuous on $(a,b)$.
Since $g_n$ is concave on $[0,1]$ (by Theorem~\ref{Th:convexity_fn}),
we conclude that $g_n$ is continuous on $(0,1)$.
Theorem~\ref{Th:f_n} implies that $g_n(c)=(n-1)c$ for $c\in[0,1/(n-1)^2]$.
Thus $g_n$ is continuous at the point $0$.
Let us show that $g_n$ is continuous at the point $1$.
We have $g_n(1)=(f_n(1))^{1/(n-1)}=1$ (Proposition~\ref{prop:f_n and c_D} implies
that $f_n(1)=1$) and $g_n(0)=0$.
Since $g_n$ is concave on $[0,1]$, we conclude that 
$g_n(c)\geqslant c$ for all $c\in [0,1]$.
Since $g_n$ is non-decreasing on $[0,1]$, we conclude that
$g_n(c)\leqslant 1$ for all $c\in[0,1]$.
Thus $c\leqslant g_n(c)\leqslant 1$ for $c\in[0,1]$.
It follows that $\lim_{c\to 1-}g_n(c)=1=g_n(1)$.
Therefore $g_n$ is continuous at the point $1$.

We proved that the function $g_n$ is continuous at every point of the segment $[0,1]$.
Thus $g_n$ is continuous on $[0,1]$.

\subsection{Proof of Theorem~\ref{Th:functional equation}}

Fix $c\in[1/(n-1)^2,1]$.
Consider arbitrary matrix $A\in\mathcal{H}_n(1+(n-1)c)$.
Then
$0\leqslant A\leqslant (1+(n-1)c)I$,
$0\leqslant (1+(n-1)c)I-A\leqslant (1+(n-1)c)I$ and
$$
0\leqslant\dfrac{(1+(n-1)c)I-A}{(n-1)c}\leqslant
\left(1+\dfrac{1}{(n-1)c}\right)I=
\left(1+\dfrac{n-1}{(n-1)^2 c}\right)I.
$$
Define 
$$
B:=\dfrac{(1+(n-1)c)I-A}{(n-1)c},
$$ 
then
$b_{ii}=1$, $i=1,2,...,n$ and
$b_{ij}=-a_{ij}/((n-1)c)$ for $i\neq j$.
It follows that $B\in\mathcal{H}_n(1+(n-1)/((n-1)^2 c))$ and
$\Pi(B)=\Pi(A)/((n-1)^{n-1}c^{n-1})$.
Since the mapping $A\mapsto B$ from 
$\mathcal{H}_n(1+(n-1)c)$ to $\mathcal{H}_n(1+(n-1)/((n-1)^2 c))$
is one-to-one and onto, we conclude that
\begin{align*}
f_n\left(\dfrac{1}{(n-1)^2 c}\right)&=
\max\{\Pi(B)|B\in\mathcal{H}_n(1+\dfrac{n-1}{(n-1)^2 c})\}=\\
&=\max\{\dfrac{\Pi(A)}{(n-1)^{n-1}c^{n-1}}|A\in\mathcal{H}_n(1+(n-1)c)\}=
\dfrac{f_n(c)}{(n-1)^{n-1}c^{n-1}}.
\end{align*}

\subsection{Proof of Proposition~\ref{P:optimal matrix}}

First assume that
$$
\dfrac{b_{12}}{a_{12}}+\dfrac{b_{23}}{a_{23}}+...+\dfrac{b_{n-1,n}}{a_{n-1,n}}\leqslant n-1
$$
for arbitrary matrix $B\in\mathcal{H}_n(1+(n-1)c)$ with $b_{i,i+1}\geqslant 0$ for $i=1,2,...,n-1$.
Then
$$
n-1\geqslant\dfrac{b_{12}}{a_{12}}+\dfrac{b_{23}}{a_{23}}+...+\dfrac{b_{n-1,n}}{a_{n-1,n}}\geqslant
(n-1)\sqrt[n-1]{\dfrac{b_{12}...b_{n-1,n}}{a_{12}...a_{n-1,n}}}.
$$
It follows that
$$
\Pi(B)=b_{12}b_{23}...b_{n-1,n}\leqslant a_{12}a_{23}...a_{n-1,n}=\Pi(A).
$$
and therefore $f_n(c)=\Pi(A)$.

Now assume that a matrix $A$ is optimal, i.e., $\Pi(A)=f_n(c)$.
Consider arbitrary matrix $B\in\mathcal{H}_n(1+(n-1)c)$ with $b_{i,i+1}\geqslant 0$ for $i=1,2,...,n-1$.
For arbitrary number $\alpha\in[0,1]$ the matrix $(1-\alpha)A+\alpha B$ belongs to $\mathcal{H}_n(1+(n-1)c)$.
Define the function
$$
\varphi(\alpha):=\Pi((1-\alpha)A+\alpha B)=((1-\alpha)a_{12}+\alpha b_{12})...((1-\alpha)a_{n-1,n}+\alpha b_{n-1,n}),
\,\alpha\in[0,1].
$$
Since $A$ is optimal, we conclude that $\varphi(\alpha)\leqslant\Pi(A)=\varphi(0)$ for $\alpha\in[0,1]$.
It follows that $\varphi'(0)\leqslant 0$, i.e.,
$$
\sum_{i=1}^{n-1}\dfrac{a_{12}a_{23}...a_{n-1,n}}{a_{i,i+1}}(b_{i,i+1}-a_{i,i+1})\leqslant 0.
$$
Thus
$$
\dfrac{b_{12}}{a_{12}}+\dfrac{b_{23}}{a_{23}}+...+\dfrac{b_{n-1,n}}{a_{n-1,n}}\leqslant n-1.
$$

\subsection{Proof of Theorem~\ref{Th:inequality for fn}}

For $n\geqslant 2$ set $D_n:=n/(4\sin^2(\pi/(2n)))$.

\begin{lemma}\label{L:difference R}
For arbitrary real numbers $a_1,...,a_n$ the following inequality holds:
\begin{equation}\label{eq:difference R}
\sum_{i<j}(a_i-a_j)^2\leqslant D_n\sum_{i=1}^{n-1}(a_i-a_{i+1})^2.
\end{equation}
\end{lemma}
\begin{proof}
Consider the inequality
\begin{equation}\label{ineq:D}
\sum_{i<j}(a_i-a_j)^2\leqslant D\sum_{i=1}^{n-1}(a_i-a_{i+1})^2,
\end{equation}
where $D>0$ and $a_1,...,a_n\in\mathbb{R}$.
We have to show that this inequality is valid for $D=D_n$ and arbitrary $a_1,...,a_n\in\mathbb{R}$.
Inequality~\eqref{ineq:D} does not change after substitution $a_i\to a_i+b$, $i=1,2,...,n$, where $b\in\mathbb{R}$.
Therefore without loss of generality we can and will assume that $a_1+...+a_n=0$.
Then the left side of inequality~\eqref{ineq:D} is equal to
\begin{align*}
&\sum_{i<j}(a_i-a_j)^2=\sum_{i<j}(a_i^2+a_j^2-2a_i a_j)=(n-1)(a_1^2+...+a_n^2)-2\sum_{i<j}a_i a_j=\\
&=n(a_1^2+...+a_n^2)-(a_1+...+a_n)^2=n(a_1^2+...+a_n^2).
\end{align*}
Thus inequality~\eqref{ineq:D} is equivalent to the inequality
$$
D\sum_{i=1}^{n-1}(a_i-a_{i+1})^2\geqslant n(a_1^2+...+a_n^2)
$$
which is equivalent to
\begin{equation}\label{ineq:n/D}
\sum_{i=1}^{n-1}(a_i-a_{i+1})^2\geqslant\dfrac{n}{D}(a_1^2+...+a_n^2).
\end{equation}
Define the matrix
$$
L=\begin{pmatrix}
1 & -1 & 0 & \cdots & 0 & 0\\
-1 & 2 & -1 & 0 &\ddots & 0\\
0 &-1 & 2 & -1& \ddots&\vdots\\
\vdots& 0 & -1& \ddots& \ddots & 0\\
0&\ddots& \ddots & \ddots &2&-1\\
0&0&\cdots &0&-1 & 1 
\end{pmatrix}
$$
corresponding to the quadratic form $\sum_{i=1}^{n-1}(a_i-a_{i+1})^2$.
The matrix $L$ is the Laplacian matrix of the graph $P_n$ with vertices $1,2,...,n$ and
edges $\{1,2\}, \{2,3\},..., \{n-1,n\}$ (the path of length $n-1$).
Let $\lambda_1\leqslant\lambda_2\leqslant...\leqslant\lambda_n$ be the spectrum of $L$.
It is clear that the eigenvalue $\lambda_1=0$ (with a corresponding eigenvector $(1,1,...,1)^t$)
and the multiplicity of $\lambda_1$ is equal to $1$.
Inequality~\eqref{ineq:n/D} can be written as $\langle La,a\rangle\geqslant(n/D)\|a\|^2$, where
a vector $a=(a_1,...,a_n)^t$ is orthogonal to the vector $(1,1,...,1)^t$.
Therefore this inequality will be valid if $n/D=\lambda_2$, i.e., if $D=n/\lambda_2$.
It is well-known that $\lambda_2=4\sin^2(\pi/(2n))$.
Thus inequality~\eqref{ineq:D} will be valid with $D=n/(4\sin^2(\pi/(2n)))=D_n$.
\end{proof}

\begin{lemma}\label{L:difference H}
For arbitrary vectors $v_1,...,v_n\in H$ the following inequality holds:
\begin{equation*}
\sum_{i<j}\|v_i-v_j\|^2\leqslant D_n\sum_{i=1}^{n-1}\|v_i-v_{i+1}\|^2.
\end{equation*}
\end{lemma}
\begin{proof}
Set $a_1:=0$ and $a_i:=\|v_1-v_2\|+...+\|v_{i-1}-v_i\|$ for $i\geqslant 2$.
For $i<j$ we have $a_i-a_j=-(\|v_i-v_{i+1}\|+...+\|v_{j-1}-v_j\|)$.
Using Lemma~\ref{L:difference R} we get
\begin{equation*}
\sum_{i<j}(\|v_i-v_{i+1}\|+...+\|v_{j-1}-v_j\|)^2\leqslant D_n\sum_{i=1}^{n-1}\|v_i-v_{i+1}\|^2.
\end{equation*}
It follows that
\begin{equation*}
\sum_{i<j}\|v_i-v_j\|^2\leqslant 
\sum_{i<j}(\|v_i-v_{i+1}\|+...+\|v_{j-1}-v_j\|)^2\leqslant D_n\sum_{i=1}^{n-1}\|v_i-v_{i+1}\|^2.
\end{equation*}
\end{proof}

Now we are ready to prove Theorem~\ref{Th:inequality for fn}.
The proof of Theorem~\ref{Th:inequality for fn} is based on Proposition~\ref{prop:f_n and c_D}.
Let $H$ be a complex Hilbert space and $H_1,...,H_n$ be closed subspaces of $H$.
Denote by $P_i$ the orthogonal projection onto $H_i$, $i=1,...,n$.
Assume that $c_D(H_1,...,H_n)\leqslant c$.
We have to prove that
$$
\|P_n...P_2 P_1\|\leqslant
\sqrt{\dfrac{n-4(n-1)(\sin^2(\pi/(2n)))(1-c)}{n+4(n-1)^2(\sin^2(\pi/(2n)))(1-c)}}.
$$
By the definition of $c_D$ for arbitrary vectors $x_1\in H_1,...,x_n\in H_n$ we have
\begin{equation*}
2\sum_{i<j}Re\langle x_i,x_j\rangle\leqslant c_D(H_1,...,H_n)(n-1)\sum_{i=1}^n \|x_i\|^2\leqslant c(n-1)\sum_{i=1}^n\|x_i\|^2.
\end{equation*}
It follows that
\begin{align*}
&\sum_{i<j}\|x_i-x_j\|^2=\sum_{i<j}(\|x_i\|^2+\|x_j\|^2-2Re\langle x_i,x_j\rangle)=\\
&=(n-1)\sum_{i=1}^n\|x_i\|^2-2\sum_{i<j}Re\langle x_i,x_j\rangle\geqslant
(n-1)\sum_{i=1}^n\|x_i\|^2-c(n-1)\sum_{i=1}^n\|x_i\|^2=\\
&=(n-1)(1-c)\sum_{i=1}^n \|x_i\|^2=:\varepsilon\sum_{i=1}^n\|x_i\|^2.
\end{align*}
By Lemma~\ref{L:difference H} we get
\begin{equation}\label{eq:difference in H}
D_n\sum_{i=1}^{n-1}\|x_i-x_{i+1}\|^2\geqslant\sum_{i<j}\|x_i-x_j\|^2\geqslant\varepsilon\sum_{i=1}^n\|x_i\|^2.
\end{equation}
Now consider arbitrary $x\in H$ and set $x_i:=P_i...P_2 P_1 x$, $i=1,...,n$.
Then $x_{i+1}=P_{i+1}x_i$, $i=1,...,n-1$.
It follows that $\|x_{i+1}\|\leqslant\|x_i\|$ and 
\begin{equation*}
\|x_i-x_{i+1}\|^2=\|x_i-P_{i+1}x_i\|^2=\|x_i\|^2-\|P_{i+1}x_i\|^2=\|x_i\|^2-\|x_{i+1}\|^2.
\end{equation*}
Thus $\|x_1\|\geqslant \|x_2\|\geqslant...\geqslant \|x_n\|$ and $\sum_{i=1}^{n-1}\|x_i-x_{i+1}\|^2=\|x_1\|^2-\|x_n\|^2$.
Using~\eqref{eq:difference in H} we get
\begin{equation*}
D_n(\|x_1\|^2-\|x_n\|^2)\geqslant\varepsilon\sum_{i=1}^n\|x_i\|^2\geqslant\varepsilon(\|x_1\|^2+(n-1)\|x_n\|^2).
\end{equation*}
We rewrite this inequality as follows:
\begin{equation*}
(D_n-\varepsilon)\|x_1\|^2\geqslant (D_n+(n-1)\varepsilon)\|x_n\|^2,
\end{equation*}
i.e.,
\begin{equation*}
\|x_n\|^2\leqslant\dfrac{D_n-\varepsilon}{D_n+(n-1)\varepsilon}\|x_1\|^2,
\end{equation*}
that is,
\begin{equation*}
\|x_n\|\leqslant\sqrt{\dfrac{D_n-\varepsilon}{D_n+(n-1)\varepsilon}}\|x_1\|.
\end{equation*}
Since $x_n=P_n...P_2 P_1 x$ and $x_1=P_1 x$, we conclude that
\begin{equation*}
\|P_n...P_2 P_1 x\|\leqslant\sqrt{\dfrac{D_n-\varepsilon}{D_n+(n-1)\varepsilon}}\|P_1 x\|\leqslant 
\sqrt{\dfrac{D_n-\varepsilon}{D_n+(n-1)\varepsilon}}\|x\|.
\end{equation*}
It follows that $\|P_n...P_2 P_1\|\leqslant \sqrt{\dfrac{D_n-\varepsilon}{D_n+(n-1)\varepsilon}}$.
Finally, note that
\begin{equation*}
\dfrac{D_n-\varepsilon}{D_n+(n-1)\varepsilon}=
\dfrac{\dfrac{n}{4\sin^2(\pi/(2n))}-(n-1)(1-c)}{\dfrac{n}{4\sin^2(\pi/(2n))}+(n-1)^2(1-c)}=
\dfrac{n-4(n-1)(\sin^2(\pi/(2n)))(1-c)}{n+4(n-1)^2(\sin^2(\pi/(2n)))(1-c)}
\end{equation*}
and the proof of Theorem~\ref{Th:inequality for fn} is complete.

\subsection{Proof of Theorem~\ref{Th:lower bound}}

The proof of Theorem~\ref{Th:lower bound} is based on Proposition~\ref{prop:f_n and c_D}.
Consider the two-dimensional Hilbert space $H=\mathbb{C}^2$.
For a number $\alpha\in\mathbb{R}$ let
$L(\alpha)=\{(\cos\alpha,\sin\alpha)^t z\,|z\in\mathbb{C}\}$
be the one-dimensional subspace spanned by the vector $(\cos\alpha,\sin\alpha)^t$.
Let $\alpha_1,...,\alpha_n$ be real numbers such that for some
$i$ and $j$ $\alpha_i\neq\alpha_j$.
For each $\tau\geqslant 0$
consider the system of one-dimensional subspaces
$H_k:=L(\alpha_k\tau)$, $k=1,...,n$.
Let us find
$$
c(\tau):=c_D(L_1(\alpha_1\tau),...,L_n(\alpha_n\tau)).
$$
By Proposition~\ref{P:c_D and P_1+...+P_n} we have
$\|P_1+...+P_n\|=1+(n-1)c(\tau)$, where $P_k$ is the orthogonal projection
onto $L(\alpha_k\tau)$, $k=1,2,...,n$.
We have
$$
P_k=\begin{pmatrix}
\cos^2(\alpha_k\tau)&\cos(\alpha_k\tau)\sin(\alpha_k\tau)\\
\cos(\alpha_k\tau)\sin(\alpha_k\tau)&\sin^2(\alpha_k\tau)
\end{pmatrix}
$$
for $k=1,2,...,n$.
Therefore
$$
P_1+...+P_n=\begin{pmatrix}
\sum_{k=1}^n \cos^2(\alpha_k\tau)&\sum_{k=1}^n \cos(\alpha_k\tau)\sin(\alpha_k\tau)\\
\sum_{k=1}^n \cos(\alpha_k\tau)\sin(\alpha_k\tau)&\sum_{k=1}^n \sin^2(\alpha_k\tau)
\end{pmatrix}
=:M(\tau).
$$
Let us find $\|P_1+...+P_n\|=\|M(\tau)\|$.
Since the matrix $M(\tau)$ is Hermitian and positive semidefinite,
we conclude that $\|M(\tau)\|$ is equal to the largest eigenvalue of $M(\tau)$.
The characteristic polynomial of $M(\tau)$ is equal to
$\lambda^2-tr(M(\tau))\lambda+\det(M(\tau))$.
It is clear that trace of $M(\tau)$ is equal to $n$.
Consider
\begin{align*}
d(\tau)&:=\det(M(\tau))=
\left(\sum_{k=1}^n \cos^2(\alpha_k\tau)\right)
\left(\sum_{k=1}^n \sin^2(\alpha_k\tau)\right)-
\left(\sum_{k=1}^n \cos(\alpha_k\tau)\sin(\alpha_k\tau)\right)^2=\\
&=\sum_{i,j}\cos^2(\alpha_i\tau)\sin^2(\alpha_j\tau)-
\sum_{i=1}^n\cos^2(\alpha_i\tau)\sin^2(\alpha_i\tau)-\\
&-2\sum_{i<j}\cos(\alpha_i\tau)\sin(\alpha_i\tau)\cos(\alpha_j\tau)\sin(\alpha_j\tau)=\\
&=\sum_{i<j}(\cos^2(\alpha_i\tau)\sin^2(\alpha_j\tau)+
\cos^2(\alpha_j\tau)\sin^2(\alpha_i\tau)-\\
&-2\cos(\alpha_i\tau)\sin(\alpha_i\tau)\cos(\alpha_j\tau)\sin(\alpha_j\tau))=\\
&=\sum_{i<j}(\cos(\alpha_i\tau)\sin(\alpha_j\tau)-
\cos(\alpha_j\tau)\sin(\alpha_i\tau))^2=\\
&=\sum_{i<j}\sin^2((\alpha_i-\alpha_j)\tau).
\end{align*}
Now we have the following equation for the eigenvalues of $M(\tau)$:
$$
\lambda^2-n\lambda+d(\tau)=0.
$$
The largest root is equal to $(n+\sqrt{n^2-4d(\tau)})/2$.
Therefore
$$
\|P_1+...+P_n\|=\dfrac{n+\sqrt{n^2-4d(\tau)}}{2}=1+(n-1)c(\tau).
$$
Now we note a few properties of the functions $c(\tau)$ and $d(\tau)$:

(1) $d(0)=0$ and $c(0)=1$;

(2) the functions $d$ and $c$ are continuous on $[0,+\infty)$;

(3) there exists $\tau_0=\tau_0(\alpha_1,...,\alpha_n)>0$ such that
$d$ is increasing on $[0,\tau_0]$.
Consequently, $c$ is decreasing on $[0,\tau_0]$.

(4) Since $\sin^2(\alpha\tau)=\alpha^2 \tau^2+O(\tau^4)$ as $\tau\to 0+$,
we conclude that
$$
d(\tau)=s_1\tau^2+O(\tau^4),\,\tau\to 0+,
$$
where $s_1=s_1(\alpha_1,...,\alpha_n)=\sum_{i<j}(\alpha_i-\alpha_j)^2$.

(5) Since $\sqrt{1+u}=1+u/2+O(u^2)$ as $u\to 0$, we conclude that
\begin{align*}
\dfrac{n+\sqrt{n^2-4d(\tau)}}{2}&=
\dfrac{n}{2}\left(1+\sqrt{1-\dfrac{4d(\tau)}{n^2}}\right)=\\
&=\dfrac{n}{2}\left(1+1-\dfrac{2d(\tau)}{n^2}+O((d(\tau))^2)\right)=\\
&=n-\dfrac{1}{n}d(\tau)+O(\tau^4)=\\
&=n-\dfrac{1}{n}(s_1\tau^2+O(\tau^4))+O(\tau^4)=\\
&=n-\dfrac{s_1}{n}\tau^2+O(\tau^4)
\end{align*}
as $\tau\to 0+$.
Thus for $c(\tau)$ we have
$$
1+(n-1)c(\tau)=n-\dfrac{s_1}{n}\tau^2+O(\tau^4),\,\tau\to 0+,
$$
i.e.,
\begin{equation}\label{eq:c(tau)}
c(\tau)=1-\dfrac{s_1}{n(n-1)}\tau^2+O(\tau^4),\,\tau\to 0+.
\end{equation}

Now consider $\|P_n...P_2 P_1\|$.
We have
\begin{align*}
&\|P_n...P_2 P_1\|=
\|P_n...P_2 (\cos(\alpha_1\tau),\sin(\alpha_1\tau))^t\|=\\
&=|\cos((\alpha_2-\alpha_1)\tau)\cos((\alpha_3-\alpha_2)\tau)
...\cos((\alpha_{n}-\alpha_{n-1})\tau)|.
\end{align*}
Thus for small enough $\tau$ we have
\begin{align*}
&\|P_n...P_2 P_1\|=\\
&=\cos((\alpha_2-\alpha_1)\tau)\cos((\alpha_3-\alpha_2)\tau)
...\cos((\alpha_{n}-\alpha_{n-1})\tau)=\\
&=\left(1-\dfrac{(\alpha_2-\alpha_1)^2}{2}\tau^2+O(\tau^4)\right)
...
\left(1-\dfrac{(\alpha_n-\alpha_{n-1})^2}{2}\tau^2+O(\tau^4)\right)=\\
&=1-\dfrac{s_2}{2}\tau^2+O(\tau^4),\,\tau\to 0+,
\end{align*}
where $s_2=\sum_{i=1}^{n-1}(\alpha_i-\alpha_{i+1})^2$.
So
\begin{equation}\label{eq:norm of product of P}
\|P_n...P_2 P_1\|=1-\dfrac{s_2}{2}\tau^2+O(\tau^4),\,\tau\to 0+.
\end{equation}

From~\eqref{eq:c(tau)} it follows that
\begin{equation}\label{eq:1-c(tau)}
1-c(\tau)=\dfrac{s_1}{n(n-1)}\tau^2+O(\tau^4)
\sim\dfrac{s_1}{n(n-1)}\tau^2,\,\tau\to 0+
\end{equation}
and
\begin{equation}\label{eq:tau^2}
\tau^2=\dfrac{n(n-1)}{s_1}(1-c(\tau))+O(\tau^4),\,\tau\to 0+.
\end{equation}
Now using Proposition~\ref{prop:f_n and c_D},
\eqref{eq:norm of product of P}, \eqref{eq:tau^2} and \eqref{eq:1-c(tau)}
we get
\begin{align*}
f_n(c(\tau))&\geqslant\|P_n...P_2 P_1\|=1-\dfrac{s_2}{2}\tau^2+O(\tau^4)=\\
&=1-\dfrac{s_2}{2}\left(\dfrac{n(n-1)}{s_1}(1-c(\tau))+O(\tau^4)\right)+O(\tau^4)=\\
&=1-\dfrac{n(n-1)}{2}\dfrac{s_2}{s_1}(1-c(\tau))+O(\tau^4)=\\
&=1-\dfrac{n(n-1)}{2}\dfrac{s_2}{s_1}(1-c(\tau))+O((1-c(\tau))^2)\geqslant\\
&\geqslant 1-\dfrac{n(n-1)}{2}\dfrac{s_2}{s_1}(1-c(\tau))-K(1-c(\tau))^2
\end{align*}
for $\tau\in(0,\tau_1]$, where 
$\tau_1=\tau_1(\alpha_1,...,\alpha_n)>0$ and $K=K(\alpha_1,...,\alpha_n)$.
Thus
\begin{equation}\label{eq:inequality for f_n(c)}
f_n(c)\geqslant 1-\dfrac{n(n-1)}{2}\dfrac{s_2}{s_1}(1-c)-K(1-c)^2
\end{equation}
for all $c\in[c(\tau_1),1]$.

Now we want to choose $\alpha_1,...,\alpha_n$ for which the value of $s_2/s_1$
is as small as possible.
Consider $s_2/s_1$.
Since the value of $s_2/s_1$ does not change under substitution
$\alpha_i\to\alpha_i+a$, $i=1,2,...,n$, $a\in\mathbb{R}$,
we can and will assume that $\alpha_1+...+\alpha_n=0$.
This equality means that the vector 
$\overline{\alpha}=(\alpha_1,...,\alpha_n)^t$ is orthogonal to
the vector $e=(1,...,1)^t$.
For such $\overline{\alpha}$ we have
\begin{align*}
s_1&=\sum_{i<j}(\alpha_i-\alpha_j)^2=
(n-1)\sum_{i=1}^n\alpha_i^2-2\sum_{i<j}\alpha_i\alpha_j=\\
&=n\sum_{i=1}^n\alpha_i^2-(\alpha_1+...+\alpha_n)^2=n\sum_{i=1}^n\alpha_i^2.
\end{align*}
Also $s_2=\langle L\overline{\alpha},\overline{\alpha} \rangle$, where
$$
L=\begin{pmatrix}
1 & -1 & 0 & \cdots & 0 & 0\\
-1 & 2 & -1 & 0 &\ddots & 0\\
0 &-1 & 2 & -1& \ddots&\vdots\\
\vdots& 0 & -1& \ddots& \ddots & 0\\
0&\ddots& \ddots & \ddots &2&-1\\
0&0&\cdots &0&-1 & 1 
\end{pmatrix}
$$
and $\langle\cdot,\cdot\rangle$ is the standard inner product in $\mathbb{R}^n$.
Note that the matrix $L$ is the Laplacian matrix of the graph 
$\mathcal{P}_n$ with vertices $1,2,...,n$ and
edges $\{1,2\}, \{2,3\},..., \{n-1,n\}$ (the path of length $n-1$).
Let $\lambda_1\leqslant\lambda_2\leqslant...\leqslant\lambda_n$ 
be the spectrum of $L$.
It is clear that $L$ is positive semidefinite and $\ker(L)$
is the one-dimensional subspace spanned by the vector $e$.
Thus $\lambda_1=0$ and $\lambda_2>0$.
Note that $\lambda_2$ is called the algebraic connectivity of 
the graph $\mathcal{P}_n$ and is denoted by $a(\mathcal{P}_n)$.
It is well-known that 
$\lambda_2=a(\mathcal{P}_n)=4\sin^2(\pi/(2n))$.

Now we return to the problem of minimizing the value of $s_2/s_1$.
We have
$$
\dfrac{s_2}{s_1}=\dfrac{1}{n}
\dfrac{\langle L\overline{\alpha},\overline{\alpha} \rangle}{\|\overline{\alpha}\|^2}.
$$
The minimum value of 
$\langle L\overline{\alpha},\overline{\alpha} \rangle/\|\overline{\alpha}\|^2$
under conditions $\langle\overline{\alpha},e\rangle=0$, $\overline{\alpha}\neq 0$
is equal to $\lambda_2$ 
(and it is attained when $\overline{\alpha}$ is
an eigenvector of $L$ corresponding to the eigenvalue $\lambda_2$).
So, let $\overline{\alpha}$ be an eigenvector of $L$ corresponding to the eigenvalue
$\lambda_2$, then from \eqref{eq:inequality for f_n(c)} it follows that
\begin{align*}
f_n(c)&\geqslant 1-\dfrac{n(n-1)}{2}\dfrac{4\sin^2(\pi/(2n))}{n}(1-c)-K(1-c)^2=\\
&=1-2(n-1)(\sin^2(\pi/(2n)))(1-c)-K(1-c)^2
\end{align*}
for all $c\in[c_n,1]$, where $c_n<1$ and $K=K_n$.
By enlarging $K$, if necessary, we get the inequality
$$
f_n(c)\geqslant 1-2(n-1)(\sin^2(\pi/(2n)))(1-c)-K(1-c)^2
$$
for all $c\in[0,1]$.


\begin{thebibliography}{99}

\bibitem{Aro50}
N. Aronszajn,
\emph{Theory of reproducing kernels},
Trans. Amer. Math. Soc.
68 (1950) 337--404.

\bibitem{BGM10}
C. Badea, S. Grivaux, V. M\"{u}ller,
\emph{A generalization of the Friedrichs angle and the method of alternating projections},
C. R. Math. Acad. Sci. Paris 348 (1-2) (2010) 53--56.

\bibitem{BGM11}
C. Badea, S. Grivaux, V. M\"{u}ller,
\emph{The rate of convergence in the method of alternating projections},
Algebra i Analiz 23 (3) (2011) 1--30.

\bibitem{BS16}
C. Badea, D. Seifert,
\emph{Ritt operators and convergence in the method of alternating projections},
J. Approx. Theory 205 (2016) 133--148.

\bibitem{Deu92}
F. Deutsch,
\emph{The method of alternating orthogonal projections}.
In: S.P. Singh (eds.) Approximation Theory, Spline Functions and Applications,
NATO ASI Series (Series C: Mathematical and Physical Sciences), vol. 356,
Springer, Dordrecht, 1992, pp. 105--121.

\bibitem{Deu95}
F. Deutsch,
\emph{The angle between subspaces of a Hilbert space}.
In: S.P. Singh (eds.) Approximation Theory, Wavelets and Applications,
NATO Science Series (Series C: Mathematical and Physical Sciences), vol. 454,
Springer, Dordrecht, 1995, pp. 107--130.

\bibitem{Hal62}
I. Halperin,
\emph{The product of projection operators},
Acta Sci. Math. (Szeged) 23 (1962) 96--99.

\bibitem{KW88}
S. Kayalar, H. Weinert,
\emph{Error bounds for the method of alternating projections},
Math. Control Signals Systems 1 (1988) 43--59.

\bibitem{NS06}
A. Netyanun, D.C. Solmon,
\emph{Iterated products of projections in Hilbert space},
Amer. Math. Monthly 113 (7) (2006) 644--648.

\bibitem{Neu50}
J. von Neumann,
\emph{Functional Operators---Vol. II. The Geometry of Orthogonal Spaces},
Princeton University Press, Princeton, 1950
(a reprint of mimeographed lecture notes first distributed in 1933).





\end{thebibliography}
\end{document}